\documentclass{amsart}

\usepackage{tipa}
\usepackage{relsize}

\usepackage{cite}

\usepackage{amsmath}
\usepackage{amsthm}

\usepackage{amssymb}

\usepackage{graphicx}
\usepackage[dvipsnames]{xcolor}

\DeclareGraphicsExtensions{.eps,.pdf,.png,.jpg}

\usepackage{caption}
\usepackage{subcaption}

\usepackage[margin=\leftmargini, font=small, labelfont=bf, labelsep=period, tableposition=top]{caption}

\usepackage{lmodern}
\usepackage{bm}

\usepackage{appendix}

\newtheorem{theorem}{Theorem}[section]
\newtheorem{lemma}[theorem]{Lemma}

\theoremstyle{definition}

\newtheorem{definition}[theorem]{Definition}

\theoremstyle{remark}

\numberwithin{equation}{section}

\theoremstyle{plain}
\newtheorem{prop}[theorem]{Proposition}

\newtheorem{theorem_without}{Theorem}
\newtheorem{theorem_without_2}{Theorem}

\newcommand*{\N}{\mathbb{N}}  
\newcommand*{\Q}{\mathbb{Q}}  
\newcommand*{\R}{\mathbb{R}}  

\newcommand*{\eps}{\varepsilon}

\newcommand*{\norm}[1]{\left\lVert#1\right\rVert} 

\newcommand\restr[2]{{
  \left.\kern-\nulldelimiterspace 
  #1 
  \littletaller 
  \right|_{#2} 
  }}

\newcommand{\littletaller}{\mathchoice{\vphantom{\big|}}{}{}{}}

\usepackage[colorlinks=true, citecolor=black]{hyperref}
\hypersetup{
    colorlinks = true,
    linkcolor = {blue}
}

\definecolor{weinrot}{rgb}{0.7,0.15,0.15}

\begin{document}

\title[{\larger{\larger{\textepsilon}}}-neighbourhoods in the plane with a nowhere-smooth boundary]{{\larger{\larger{\textepsilon}}}-neighbourhoods in the plane \\ with a nowhere-smooth boundary}



%

\author[Lamb]{{Jeroen S.W.} Lamb}
\address{Jeroen S.W.~Lamb \\ Department of Mathematics \\ Imperial College London \\ 180 Queen's Gate, London SW7 2AZ, United Kingdom}
\email{jsw.lamb@imperial.ac.uk}

\author[Rasmussen]{{Martin} Rasmussen}
\address{Martin Rasmussen \\ Department of Mathematics \\ Imperial College London \\ 180 Queen's Gate, London SW7 2AZ, United Kingdom}
\email{m.rasmussen@imperial.ac.uk}

\author[Timperi]{Kalle G.~Timperi}
\address{Kalle Timperi \\ Center for Ubiquitous Computing \\ University of Oulu \\ Erkki Koiso-Kanttilan katu 3, door E, Oulu, Finland}
\email{kalle.timperi@gmail.com}

\subjclass[2010]{Set-valued and variational analysis; Differential Geometry, General Topology}

\date{\today}

\begin{abstract}
We give an example of a planar set $E\subset \mathbb{R}^2$ for which the boundary $\partial E_\eps$ of its $\varepsilon$-neighbourhood $E_\eps = \{x \in \R^2 \, : \, \textrm{dist}(x, E) \leq \eps \}$ is nowhere $C^1$-smooth, in the sense that 
the set of singularities on the boundary is countably dense (where we note that the latter set cannot be uncountable). Furthermore, we give an example of a planar set $E$ for which $\partial E_\eps$ has the same properties as above, but in addition contains an uncountable subset, with non-integer Hausdorff dimension, where curvature is not defined. Both constructions make use of a characterisation of those star-shaped sets that are an $\eps$-neighbourhood of one of their subsets.
\end{abstract}

\maketitle

\tableofcontents

\setcounter{section}{0}
\section{Introduction}
For a given set $E \subset \R^2$ and radius $\eps > 0$, the (closed) \emph{$\eps$-neighbourhood} of $E$ is the set
\begin{equation} \label{Def_Tubular_Neighbourhood}
E_\eps := \overline{B_\eps(E)} := \overline{ \bigcup_{x \in E} B_\eps(x)},
\end{equation}
where the overline denotes closure and $B_\eps(\cdot)$ is an open ball of radius $\eps$ in the Euclidean metric. The sets $E_\eps$ are sometimes called 
\emph{tubular neighbourhoods} \cite{Fu_Tubular_neighborhoods}, \emph{collars} \cite{Przeworski_An_Upper_Bound} or \emph{parallel sets}~\cite{Rataj_Winter_On_Volume, Stacho_On_the_volume}. 
%
Understanding the properties of $\eps$-neighbourhoods is a natural and fundamental question in (Euclidean) geometry, but they are particularly relevant in specific settings where $\eps$-neighbourhoods arise naturally. These include the classification and bifurcation of minimal invariant sets in random dynamical systems with bounded noise~\cite{TopoBif_of_MinInvSets}, as well as control theory~\cite{Colonius_Kliemann_Dynamics_of_Control}.

\begin{figure}[h]
      \captionsetup{margin=0.75cm}
      \centering     
                \includegraphics[width=0.75\textwidth]{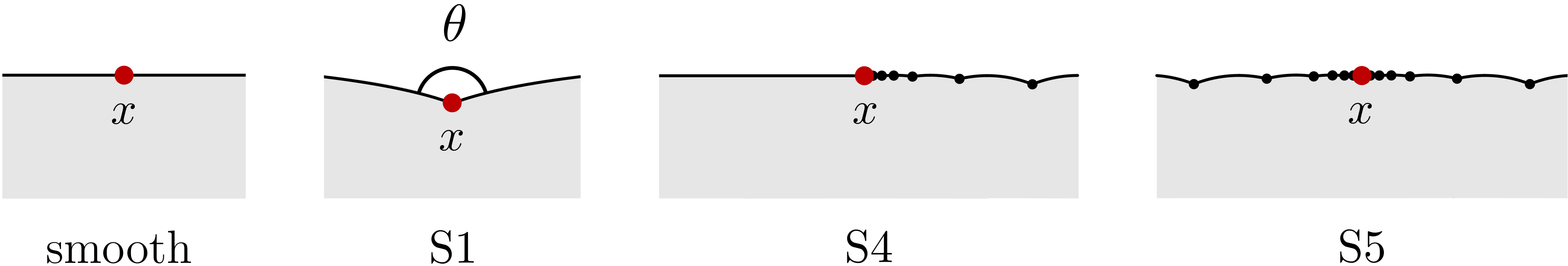} \\[-1mm]
                \caption{Possible local boundary geometries around an $\eps$-neighbourhood boundary point $x \in \partial E_\eps$ (red dot). At a smooth point (smooth) the boundary is locally a $C^1$-curve. At a wedge singularity (S1) the tangential directions at the boundary form an angle $\theta \in (0, \pi)$. One-sided (A4) and two-sided (S5) shallow singularities are accumulation points of wedge singularities, from one or two directions, respectively.}
\label{Fig_smooth_wedge_and_shallow}
\end{figure}

Intuitively, the effect of replacing a set $E$ by its $\eps$-neighbourhood amounts to smoothening and mollifying the original set. For example, $E$ might have a fractal boundary with infinite length (infinite one-dimensional Hausdorff measure) and tangents and curvature may not exist in the classical sense. By contrast, the boundary $\partial E_\eps$ of the $\eps$-neighbourhood of any set $E$ is a Lipschitz manifold for most values of $\eps$, and the set of radii for which this property does not hold has been shown to be small~\cite{Fu_Tubular_neighborhoods, Rataj_Zajicek_Smallness}. Even in the most pathological cases, the neighbourhoods of singularities on $\partial E_\eps$ can be classified to a small number of distinct categories, characterised by relatively simple topological and tangential properties~\cite{lamb2025boundariesvarepsilonneighbourhoodsplanarsets}. 
Indeed, it is challenging to come up with an $\eps$-neighbourhood whose boundary does not contain a sizeable $C^1$-smooth subset. This poses the question: how rugged can the boundary of an $\eps$-neighbourhood of a set be? In this paper we investigate this question for compact planar sets $E \subset \R^2$.

\subsection{Main Results} 
Our main result is the construction of a set $E \in \R^2$ with the property that the boundary $\partial E_\eps$ of the $\eps$-neighbourhood $E_\eps := \overline{B_\eps(E)}$ consists entirely of wedges and their accumulation points, shallow singularities. Consequently, the boundary has no subsets that can be represented as $C^1$-smooth curves, even though tangents do exist on the boundary outside a countable subset. See Figure~\ref{Fig_smooth_wedge_and_shallow} and \cite[Definition 4.1]{lamb2025boundariesvarepsilonneighbourhoodsplanarsets} for the definitions and illustrations of smooth points, wedges, and shallow singularities. 
Our main result is the following.
\begin{theorem_without}[\textbf{Existence of a nowhere-smooth boundary}] \label{Thm_main_existence}
There exists a planar set $E \subset \R^2$ and $\eps > 0$ for which the boundary $\partial E_\eps$ consists only of wedges and shallow singularities.
In particular, no subset of $\partial E_\eps$ is diffeomorphic to a $C^1$-smooth curve.
\end{theorem_without}
In addition, we construct an example of an $\eps$-neighbourhood with the property that curvature does not exist in an uncountable subset of the boundary.
This set has zero one-dimensional Hausdorff measure, in accordance with the results obtained in~\cite{lamb2025boundariesvarepsilonneighbourhoodsplanarsets}. In both examples, we use the Cantor set as the technical starting point of the construction.

\subsection{Context} It is known that for a fixed set $E \subset \R^2$, the boundary $\partial E_\eps$ is a Lipschitz manifold for almost all $\eps > 0$. Furthermore, the set of radii $\eps$ for which this property does not hold has been shown to be small in terms of entropy dimension~\cite{Fu_Tubular_neighborhoods} and more refined measures such as gap sums~\cite{Rataj_Zajicek_Smallness}. However, these results by their nature do not reveal the structure of singularities and the existence of tangents and curvature on the boundary.
Recently, the finer geometry of the boundary was systematically studied in~\cite{lamb2025boundariesvarepsilonneighbourhoodsplanarsets}. There, a complete classification of singularities 
on the boundary was established,
but it remains unaddressed whether there exists a set $E$ whose $\eps$-neighbourhood boundary $\partial E_\eps$ for some $\eps > 0$ contains no smooth points at all. A further question concerns the existence of second-order smoothness, i.e.~curvature. It was shown in~\cite{lamb2025boundariesvarepsilonneighbourhoodsplanarsets} 
that curvature does exist almost everywhere on the boundary with respect to the one-dimensional Hausdorff measure. Notwithstanding this result, it has been an open question whether it is possible for the boundary $\partial E_\eps$ to have uncountably many points at which curvature is not defined in the classical sense.

\section{ 
A boundary without smooth points}

In this section we construct a compact planar set $E$ for which the boundary $\partial E_\eps$ of the $\eps$-neighbourhood $E_\eps$ consists entirely of non-smooth points for a certain $\eps > 0$. However, we begin from the opposite direction and characterise the boundaries $\partial E_\eps$ which consist entirely of smooth points. While interesting in its own right, this characterisation is also utilised in the intermediate stages of the main construction. 

\subsection{Smoothness in terms of metric projections}
\label{Sec_Smoothness}
We now define smooth points on the boundary $\partial E_\eps$, and characterise these in terms of metric projections onto the underlying set $E$. We then introduce the notion of positive reach and use it to characterise those boundaries $\partial E_\eps$ which consist entirely of smooth points.
\begin{definition}[Smooth point, singularity]
\label{Def_smooth_singularity}
Let $E \subset \R^2$ and let $\eps > 0$. A boundary point $x \in \partial E_\eps$ is \emph{smooth} if there exists a $C^1$-curve $\Gamma$ for which $\Gamma = \partial E_\eps \cap \overline{B_\delta(x)}$ for some $\delta > 0$. If $x$ is not smooth, it is a \emph{singularity}.
\end{definition}
A complete characterisation of boundary singularities of $\eps$-neighbourhoods was provided in~\cite{lamb2025boundariesvarepsilonneighbourhoodsplanarsets}.
It turns out that smooth points on $\partial E_\eps$ can be characterised in terms of metric projections. For each $z \in \R^2$, define the \emph{metric projection} onto $E$ by
\begin{equation} \label{Def_metric_projection}
\Pi_{E}(z) := \{y \in \partial E \, : \, \textrm{dist}(z, E) = \norm{z - y} \},
\end{equation}
and consider the set
\begin{equation}
\mathrm{Unp}(E) := \{z \in \R^2 \, : \, \Pi_{E}(x) \, \textrm{is a singleton} \}.
\end{equation}
The set $\mathrm{Unp}(E)$ consists of all those points $z \in \R^2$ that have a unique metric projection onto $E$ (a unique nearest point on $\partial E$).
According to~\cite[Proposition 4.6]{lamb2025boundariesvarepsilonneighbourhoodsplanarsets}, a boundary point $x \in \partial E_\eps$ is smooth in the sense of Definition~\ref{Def_smooth_singularity} if and only if there exists some $r > 0$ for which 
$\partial E_\eps \cap B_r(x) \subset \mathrm{Unp}(E)$.\footnote{Since the set $E$ is assumed to be closed, it follows that $\partial E \subset E$ and $\Pi_{E}(z)$ is non-empty for all $z \in \R^2$.} An immediate consequence of this is that the boundary $\partial E_\eps$ consists entirely of smooth points if and only if $E_\eps \subset \mathrm{Unp}(E)$. This observation can be expressed conveniently through the concept of \emph{positive reach}. The following is a rephrasing of\cite[Definition 4.1]{Federer_Curvature_measures}.


\begin{definition}[Reach of a set] \label{Def_Positive_Reach}
Let $E \subset \R^d$ and let $y \in E$. The (local) reach of $E$ at $y$ is
\begin{equation} \label{Eq_Def_Reach}
\mathrm{reach} (E, y) := \sup\{\eps \geq 0 \, : \, B(y,\eps) \subset \mathrm{Unp}(E)\}.
\end{equation}
The \emph{reach} of the set $E \subset \R^d$ is
\begin{equation} \label{Eq_Def_Positive_Reach}
\mathrm{reach}(E) := \inf_{y \in E} \mathrm{reach} (E, y).
\end{equation}
The set $E$ has \emph{positive reach}, if $\mathrm{reach}(E) > 0$.
\end{definition}

Thus, we obtain the following characterisation for $\eps$-neighbourhoods whose boundaries consist entirely of smooth points. 
The result follows immediately from~\cite[Proposition 4.6]{lamb2025boundariesvarepsilonneighbourhoodsplanarsets} and the definitions, and we omit the simple proof.

\begin{prop}[\textbf{Characterisation of everywhere smooth boundary}] \label{Prop_Characterisation_of_smooth_boundaries}
Let $E \subset \R^2$ and let $\eps > 0$. Then the following are equivalent:
\begin{itemize}
  \item[(i)] Every $x \in \partial E_\eps$ is a smooth point,
  \item[(ii)] $E_\eps \subset \mathrm{Unp}(E)$,
  \item[(iii)] $\mathrm{reach}(E) \geq \eps$.
\end{itemize}
\end{prop}

\subsection{Construction of a nowhere smooth boundary} \label{Sec_Construction}
We now present the construction for an $\eps$-neighbourhood boundary without smooth points. There are two main difficulties in this construction.
Firstly, since the $\eps$-neighbourhood consists of $\eps$-radius balls with constant curvature, the outer boundary of any $\eps$-neighbourhood must contain subsets that curve in the same direction (although the curvature can be less than $1/\eps$). However, at each wedge singularity, the boundary tilts in the opposite direction. If wedges are dense on the boundary, it seems unlikely for the boundary to be curving along the direction implied by the circular geometry of the $\eps$-balls. However, this turns out not to be an issue. This is because there are only countably many wedges, and the angles between the tangential directions at each wedge can be chosen to be small enough to constitute a summable sequence. Secondly, instead of directly defining the underlying set $E \subset \R^2$, we first define a curve $\Gamma$ that we subsequently aim to identify as the $\eps$-neighbourhood boundary $\partial E_\eps$ of some set $E \subset \R^2$. Thus, a non-trivial step of the proof is to show that there indeed exists some set $E$ for which $\Gamma = \partial E_\eps$.

The construction proceeds in two steps.
In the first step we construct a continuous function $S : [0, 2\pi] \to \R$ with $S(0) = S(2\pi)$ as the sum $S = P + I$, where $P$ is a smooth function with bounded curvature, and $I$ is a convex function that is non-differentiable at all rationals on $[0, 2\pi]$. In the second step, we apply a polar-coordinate transform to this function to wrap it around the origin.

\textbf{Step 1.} We begin 
by constructing 
an increasing, bounded function $f: [0, 2\pi] \to \R_+$ 
that is discontinuous at every rational number $q \in \Q \cap (0,2\pi)$ but continuous at every irrational number $r \in [0,2\pi] \setminus \Q$. The construction follows~\cite[Remark 4.31]{Rudin-principles}. 

Write $\Q \cap [0,2\pi] = \big \{ q_n \, : \, n \in \N \big \}$ and define
\begin{equation} \label{Def_Set_N(x)}
N(x) := \{n  \, : \, q_n \leq x\}.
\end{equation}
Then, take any positive summable sequence $(m_n)_{n=0}^\infty$ with $\sum_{n=0}^\infty m_n = V \in (0, \infty)$ and set 
\begin{equation} \label{Eq_Def_f}
f(x) := \sum_{n \in N(x)} m_n
\end{equation}
for all $x \in [0,2\pi]$. Note that $f(0) = m_{n(0)} > 0$ for some index $n(0) \in \N$ that satisfies $q_{n(0)} = 0$.
%
%
Thus, $f$ is increasing and upper semi-continuous on $[0,2\pi]$, with a jump of amplitude $m_n$ at each rational $q_n \in (0, 2\pi]$. Furthermore, $f$ is continuous at every irrational $x \in [0, 2\pi] \setminus \Q$ and satisfies $f(2\pi) = \sum_{n=0}^\infty m_n = V$.

As an almost everywhere continuous bounded function, $f$ is Riemann-integrable. Since $f$ is monotonically increasing on $[0, 2\pi]$, the integral function
\begin{equation} \label{Eq_Def_I_f}
I(x) := \int_0^x f(s) ds
\end{equation}
is convex. Most significantly for our example, $I$ has a well-defined derivative at every irrational $r \in [0,2\pi] \setminus \Q$, but not at any rational $q \in \Q \cap (0,2\pi]$.

To obtain a curve that can be interpreted as the boundary of an $\eps$-neighbourhood, we now add to $I$ pointwise another function $P$. This function needs to be smooth, have bounded curvature, and satisfy certain boundary conditions. These are defined by the corresponding boundary conditions for the sum function $S: [0, 2\pi] \to \R$, defined pointwise by
\begin{equation} \label{Eq_Def_S}
S(x) := I(x) + P(x).
\end{equation}
Let $L := I(2\pi)$ and let $V := f(2\pi) = \sum_{n=0}^\infty m_n$ be as above. We want the image of $S$ to be a Jordan curve when mapped around the origin with the polar-coordinate transform. Thus, we require $S(0) = S(2\pi) = L$. Furthermore, in order to have a wedge singularity at the image of $0$, the right derivative of $S$ at $0$ should differ from the left derivative of $S$ at $2\pi$. This can be ensured by requiring that
\[
S'(0) = f(0) = m_{n(0)}, \quad \textrm{and} \quad S'(2\pi) = 0.
\]
Thus, the function $P$ needs to satisfy
\begin{equation} \label{Eq_Polynomial_Boundary_Conditions}
P(0) = L, \quad P'(0) = 0, \quad P(2\pi) = 0, \quad \textrm{and} \quad P'(2\pi) = -V.
\end{equation}
We have not required that $P''(0) \leq 0$, so $P$ can be increasing in some neighbourhood of zero.

\textbf{Step 2.} Consider the polar coordinate map $\alpha : \R \times \R_+ \to \R^2$ given by
\begin{equation} \label{Eq_Def_polar_map}
\alpha(x, r) := r(\cos x, \sin x)
\end{equation}
and let $J = \alpha\big(\Gamma_S \big)$ be the image under $\alpha$ of the graph $\Gamma_S$ of $S$. In other words, each angle $x \in [0, 2\pi]$ corresponds to a unique point $\alpha(x, S(x)) \in J$ which is at an angle $x$ relative to the positive horizontal axis and whose distance from the origin is $S(x)$.
It follows from $S(0) = S(2\pi)$ that $J$ is a Jordan curve. We show in the next section that 
$J$ can be interpreted as the boundary $\partial E_\eps$ of an $\eps$-neighbourhood of a certain set $E \subset \R^2$. 

\begin{figure}[h]
      \captionsetup{margin=0.75cm}
      \centering      \vspace{-1mm}
                \begin{subfigure}{0.48\textwidth} \vspace{0mm}
                \includegraphics[width=1.05\textwidth]{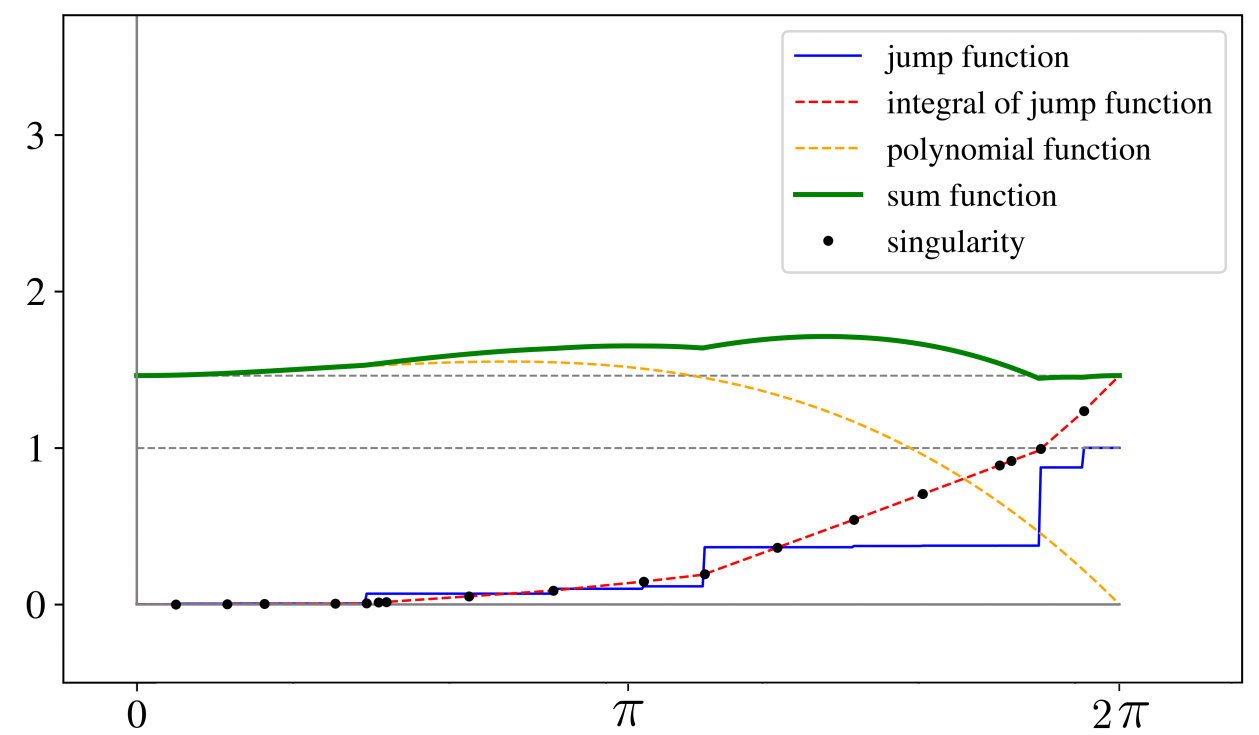} \\[-1mm]
                \caption{}
                \end{subfigure} \quad 
                \begin{subfigure}{0.48\textwidth} 
                \includegraphics[width=1.05\textwidth]{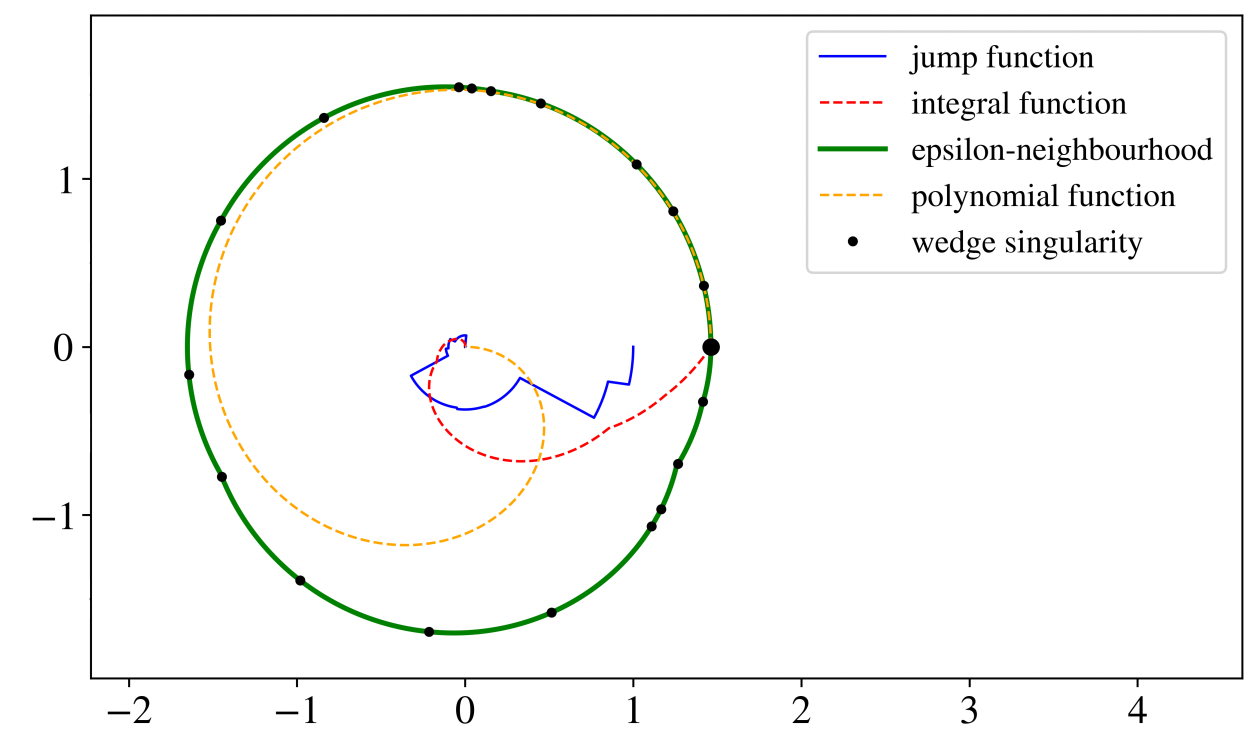} \\[-1mm] 
                \caption{}
                \end{subfigure}
                \caption{\textbf{(a)} The integral $I$ (red dashed line) of the jump function $f$ (blue line) has a singularity at each rational $q \in \Q \cap (0, 2\pi]$. The smooth function $P$ with bounded curvature (yellow dashed line) satisfies $P(0) = I(2\pi) = L$. The sum function $S = P + I$ (dark green) satisfies $S'(0) = P'(0) + I'(0) = m_{n(0)} > 0 = S'(2\pi)$.   
                \textbf{(b)} In polar coordinates $(x, r)$, the $\eps$-neighbourhood boundary (dark green curve) is defined at each angle $x \in [0, 2\pi]$ by the radius
                $r(x) = P(x) + I(x)$. Wedges and shallow singularities are dense on the boundary. The black dots indicate the positions of the eighteen singularities with largest differences between the one-sided derivatives.} \label{Figure_Functions_and_Eps_Boundary}
\end{figure}

\subsection{Numerical visualisation}
For the sake of visualising the $\eps$-neighbourhood boundary $\alpha(\Gamma_S)$ corresponding to the graph $\Gamma_S$ of $S$, we now provide an example with explicit values for the sequence $(m_n)_{n=0}^\infty$ and a concrete smooth function $P$ in~\eqref{Eq_Def_S}. For each $n \in \N$, define $m_n := 2^{-n}$ so that $V = \sum_{n \in \N} m_n = 1$. Define $I$ as in~\eqref{Eq_Def_I_f}.\footnote{Note that the function $I$ (and consequently also the functions $P$ and $S$) depends on the particular enumeration $(q_n)_{n=0}^\infty$ of the rationals via~\eqref{Def_Set_N(x)}.} We now implicitly define a third degree polynomial $P: [0, 2\pi] \to \R$ by requiring that the sum function $S = I + P$ satisfies $S(0) = S(2\pi) = I$ as well as $S'(0) = f(0) = m_{n(0)}$ and $S'(2\pi) = 0$.
These equations imply
\begin{equation} \label{Eq_Polynomial_Boundary_Conditions}
P(0) = I, \quad P'(0) = 0, \quad P(2\pi) = 0, \quad \textrm{and} \quad P'(2\pi) = -V.
\end{equation}
The polynomial $P$ is completely determined by these properties as soon as $V$ and $I$ are fixed. It is then given by the expression
\begin{equation} \label{Eq_Def_P}
P(x) = \frac{I - \pi V}{4 \pi^3}x^3 + \frac{2V\pi - 3I}{4 \pi^2} x^2 + I.
\end{equation}
We have not required that $P''(0) \leq 0$, so $P$ can be increasing in some neighbourhood of zero. See Figure~\ref{Figure_Functions_and_Eps_Boundary} for a graphical illustration of the functions $f$, $I$, $P$, and $S$, as well as the corresponding $\eps$-neighbourhood boundary as the image $\alpha(\Gamma_S)$.

\subsection{Construction of the underlying set}
In section~\ref{Sec_Construction} we defined the Jordan curve $J \subset \R^2$ that we claimed is the boundary of an $\eps$-neighbourhood $E_\eps$ for some $E \subset \R^2$. It remains to be shown that such a set $E$ exists. The Jordan curve $J$ is by construction a star-shaped set (see Definition~\ref{Def_Star_shaped_set}) given by an image under the polar-coordinate transformation. We thus need a criterion ensuring that such sets are $\eps$-neighbourhoods of some underlying set. This criterion is provided in Proposition~\ref{Prop_Star-shaped_eps-neighbourhood} and it relies on Theorem~\ref{Thm_preserving_Pos_Reach}, due to Federer. 

%


\begin{theorem}[\textbf{Positive reach is preserved in bi-Lipschitz maps} \hspace{-0.1mm}{\cite[Theorem 4.19]{Federer_Curvature_measures}}] \label{Thm_preserving_Pos_Reach}
Let $A \subset \R^d$ and $\textrm{reach}(A) > t > 0$. In addition, let $s > 0$ and let
\[
\alpha: \{x \, : \, \mathrm{dist}(x, A) < s\} \to \R^d
\]
be an injective, continuously differentiable map such that $\alpha$, $\alpha^{-1}$, $D\alpha$ are Lipschitz, with respective constants $M, N$ and $P$. Then
\[
\mathrm{reach}[\alpha(A)] \geq \inf \big\{ sN^{-1}, (Mt^{-1} + P)^{-1} N^{-2})\big\}.
\]
\end{theorem}
Essentially, Theorem~\ref{Thm_preserving_Pos_Reach} says that the property of a set $A$ having positive reach is preserved in mappings that can be defined in some $s$-neighbourhood of $A$ and are sufficiently smooth. The reach of the image will typically be different than that of the original set, but it remains positive.


\begin{lemma} \label{Lemma_Positive_reach_of_images_of_balls}
Let $a,b > 0$, let $0 < \delta < \min\{a, \pi\}$ and let $\alpha$ be the polar coordinate transformation~\eqref{Eq_Def_polar_map}. Then there exists some $\eps > 0$ such that $\mathrm{reach}(\alpha ( \partial B_\delta)) \geq \eps$ for all $\delta$-radius balls $\overline{B_\delta} \subset \R \times [a,b]$. It follows that for each such ball there exists some $E \subset \alpha(\overline{B_\delta})$ for which $\alpha (\overline{B_\delta}) = \overline{\bigcup_{y \in E} B_\eps(y)}$.
\end{lemma}

\begin{proof}
The result is a direct application of Theorem~\ref{Thm_preserving_Pos_Reach} once we check that the assumptions therein hold. Clearly, the reach of the circle $\partial B_\delta$ is defined by the distance from the boundary to the center of the circle, so $\mathrm{reach}(\partial B_\delta) = \delta$. Since $\alpha$ is periodic in the first argument, it follows that $\alpha$, $\alpha^{-1}$ and $D\alpha$ have bounded Lipschitz constants in the set $\R \times [a,b]$. For $s := (\min\{a, \pi\} - \delta) / 2$, the function $\alpha$ is injective in the $s$-neighbourhood of $\partial B_\delta$, and the aforementioned Lipschitz constants are clearly bounded also in this set. It then follows from Theorem~\ref{Thm_preserving_Pos_Reach} that there exists some $\eps > 0$ for which $\mathrm{reach}(\alpha(\partial B_\delta)) \geq \eps$.
It remains to prove the second claim. Let $\eps$ be as above and define
$E := \big\{ z \in \alpha\big(\overline{B_\delta}\big) \, : \, \mathrm{dist}(z,  \alpha\big(\partial B_\delta \big) \big) \geq \eps \big\}$. Clearly $E \subset \mathrm{int}\textrm{ }\alpha\big(\overline{B_\delta}\big)$ and since $\mathrm{reach}(\alpha ( \partial B_\delta)) \geq \eps$ it follows that $\overline{B_\eps(y)} \subset \alpha\big(\overline{B_\delta}\big)$ for all $y \in E$. Then $\alpha (\overline{B_\delta}) = \overline{\bigcup_{y \in E} B_\eps(y)}$, which was to be shown.
\end{proof}

We now recall the definition of star-shaped sets.
\begin{definition}[Star-shaped set] \label{Def_Star_shaped_set}
A subset $A$ of a vector space is \emph{star shaped}, if there exists some $a_0 \in A$ for which the line segment $[a_0, x]$ between $a_0$ and $x$ is contained in $A$ for every $x \in A$. The collection of points $a_0$ with this property is called the \emph{kernel} of $A$, and is denoted $\textrm{Ker}(A)$.
\end{definition}
Let $X$ be a set and $Y$ an ordered set. The \emph{epigraph} and \emph{hypograph} of a function $f : X \to Y$ are the sets
\[
\mathrm{epi} f := \big\{(x, y) \in X \times Y \, : \, y > f(x) \big\} \qquad
\mathrm{hypo} f := \big\{(x, y) \in X \times Y \, : \, y < f(x) \big\}.
\]
Geometrically, these sets correspond to the areas above and below the graph of $f$, respectively.

\begin{prop} \label{Prop_Star-shaped_eps-neighbourhood}
Let $A \subset \R^2$ be a closed star-shaped set whose kernel has a non-empty interior $\mathrm{int}\textrm{ }\mathrm{Ker}(A) \neq \varnothing$. Let $\alpha$ be the polar coordinate transformation~\eqref{Eq_Def_polar_map}. Then $A$ is an $\eps$-neighbourhood $\overline{B_\eps(E)}$ of some set $E \subset A$ if and only if there exists a $2\pi$-periodic function $f: \R \to \R_+$ for which 
\begin{itemize}
 \item[(i)] $\mathrm{min}_{x \in [0, 2\pi]} f(x) > 0$,
 \item[(ii)] $\mathrm{reach}(\mathrm{epi} f) > 0$, and
 \item[(iii)] $A = a + \alpha(\mathrm{hypo} f)$ for some $a \in \mathrm{int}\textrm{ }\mathrm{Ker}(A)$.
\end{itemize}
\end{prop}

\begin{proof}
\textbf{$\Rightarrow$:} Assume first that $A$ is a star-shaped set whose kernel contains an open set, and that the function $f$ described in the proposition exists.
To show that $A$ is an $\eps$-neighbourhood of some set $E$, it suffices to find a set $G$ and $\eps > 0$ for which $\overline{B_\eps(G)} \subset A$ and  $\partial A \subset \partial B_\eps(G)$. The required set $E$ can then be obtained by adding to $G$ any interior points of $A$ whose distance to the boundary is greater than $\eps$.

Clearly $\partial A = \alpha (\Gamma_f)$ where $\Gamma_f$ is the graph of $f$. Since $\mathrm{reach} (\mathrm{epi} f) > 0$, it follows from~\cite[Theorem 6.4]{Rataj_Zajicek_On_the_Structure_of} that $f$ is Lipschitz. Therefore, $f$ has at every point $x \in [0, 2\pi]$ the one-sided derivatives $\partial_-f(x), \partial_+f(x)$ and corresponding left and right unit normal vectors $\overline{n}^\pm(x)$, in the direction of the hypograph of $f$. Let $m := \min_{x \in [0, 2\pi]}\{f(x)\}$ and $M := \max_{x \in [0, 2\pi]}\{f(x)\}$
and define
\[
\delta := \min\{\mathrm{reach} (\mathrm{epi} f), m, \pi\} / 4.
\]
Then, for all $x \in [0, 2\pi]$, the $\delta$-radius balls $\overline{B_\delta\big(z_x^\pm\big)}$ centered at $z_x^\pm := (x, f(x)) + \delta \overline{n}^\pm(x)$ satisfy
\begin{equation}
\overline{B_\delta\big(z_x^\pm\big)} \subset \big\{ (x, y) \in \mathrm{hypo} f \, : y \geq m/2 \big\}.
\end{equation}
In particular, each such ball is contained in the set $\R \times [m/2,M]$. Thus, it follows from Lemma~\ref{Lemma_Positive_reach_of_images_of_balls} that there exists some $\eps > 0$ such that for all $x \in [0, 2\pi]$ we have
\[
\alpha\big(\overline{B_\delta\big(z_x^\pm\big)}\big) = \overline{\bigcup_{y \in E_x^\pm} B_\eps(y)}
\]
for some $x$-dependent sets $E_x^\pm \subset \mathrm{int}\textrm{ }\overline{\alpha\big(B_\delta\big(z_x^\pm\big)\big)}$. Let $G := \bigcup_{x \in [0, 2\pi]} E_x^\pm$. Then clearly $\overline{B_\eps(G)} \subset A$ and $\partial A \subset \partial B_\eps(G)$ from which the result follows.

\textbf{$\Leftarrow$:} We sketch the proof of the opposite direction, which we will not need for the subsequent argument. Assume that $A$ is a star-shaped set whose kernel contains an open set, and that $A = \overline{B_\eps(E)}$ for some set $E$. Without loss of generality we can assume that $0 \in \mathrm{int}\textrm{ }\mathrm{Ker}(A)$. Since $A$ is star shaped, every point $z \in \partial A$ can be connected to $0$ through a line segment. The lengths of these line segments then define a function $f : [0, 2\pi] \to \R_+$ which can be extended into a $2\pi$ periodic function on the real line. Since $\textrm{ }\mathrm{Ker}(A)$ contains an open set, this $f$ satisfies $\mathrm{min}_{x \in [0, 2\pi]} f(x) > 0$. To show that $\mathrm{reach}(\mathrm{epi} f) > 0$, one can follow the approach above by considering the inverse images under the map $\alpha$ of $\eps$-balls $B_\eps$ touching the boundary $\partial A$ and using Theorem~\ref{Thm_preserving_Pos_Reach} to obtain some $\delta > 0$ for which $\mathrm{reach}\big(\alpha^{-1}(\partial B_\eps)\big) > \delta$. 
One can then show $f$ to be semiconvex due to the fact that wedge singularities correspond to jump discontinuities of the derivative $f'$ in only the positive direction, and the curvature of $f$ (where defined) is bounded from below by $1/\delta$. It then follows from~\cite[Theorem 2.3]{Fu_Tubular_neighborhoods} that $\mathrm{reach}(\mathrm{epi} f) > 0$.
\end{proof}

We now proceed to prove our main result by applying Proposition~\ref{Prop_Star-shaped_eps-neighbourhood} to the construction given in Section~\ref{Sec_Construction}.

\begin{theorem_without_2}[\textbf{Existence of a nowhere-smooth boundary}] \label{Thm_main_existence}
There exists a planar set $E \subset \R^2$ and $\eps > 0$ for which the boundary $\partial E_\eps$ consists only of wedges and shallow singularities.
In particular, no subset of $\partial E_\eps$ is diffeomorphic to a $C^1$-smooth curve.
\end{theorem_without_2}

\begin{proof}
Let $f$, $I$, $S$, and $P$ be as in~\eqref{Eq_Def_f}, \eqref{Eq_Def_I_f}, and~\eqref{Eq_Def_S}, respectively, and let $\alpha$ be the polar coordinate map in~\eqref{Eq_Def_polar_map}. Consider the image $A := \alpha(\mathrm{hypo}\textrm{ }S)$ and the corresponding boundary $J = \partial A = \alpha\big(\Gamma_S\big)$ where
\begin{equation} \label{Eq_Non-smooth_Jordan_curve}
\Gamma_S := \big\{\big(x, S(x)\big) : x \in [0, 2\pi] \big\}
\end{equation}
is the graph of $S$. Since $S$ has singularities at every rational argument, and these are mapped onto wedge singularities by the transformation $\alpha$, it suffices to show that $J = \partial B_\eps(E)$ for some $E$. Proposition~\ref{Prop_Star-shaped_eps-neighbourhood} guarantees the existence of such a set $E$ once we demonstrate that $S$ satisfies the conditions (i)--(iii) listed therein. Out of these, (i) and (iii) follow directly from the definitions of $S$ and $A$. We now show that $\mathrm{epi}\textrm{ }S$ is a set with positive reach.

According to\cite[Theorem 2.3]{Fu_Tubular_neighborhoods}, the epigraph of a locally Lipschitzian function $f$ is a set with positive reach if $f$ is semiconvex. A function $f$ is semiconvex if there exists $C > 0$ such that the function $g(x) := f(x) + (C/2)x^2$ is convex~\cite[Definition 2.1]{Rataj_Zajicek_On_the_Structure_of}. Intuitively, this means that even if $f$ is not convex, it can be made convex by adding to it pointwise a sufficiently steeply increasing second-degree polynomial $P_C(x) := (C/2)x^2$. In our case, the function $I$ is convex by construction, and the smooth function $P$ was assumed to have bounded curvature on $[0, 2\pi]$. The second derivative of $P + P_C$ can thus be made positive by choosing a sufficiently large $C > 0$. Therefore, $S$ is semiconvex and $\mathrm{reach}(\mathrm{epi}\textrm{ }S) > 0$.
\end{proof}

\section{Absence of Curvature in a Set with Positive Hausdorff Dimension} \label{Sec_Absence_of_curvature}
In this section we augment the construction given in Section~\ref{Sec_Construction} and produce an example of an $\eps$-neighbourhood boundary that has the properties described in Theorem~\ref{Thm_main_existence}, but in addition, contains a subset of positive Hausdorff dimension in which curvature is not defined in the classical sense. We construct the example by defining on the interval $[0, 2\pi]$ a real-valued function that we then add pointwise to the analogous function $I$ defined in Section~\ref{Sec_Construction}. We then verify that this addition preserves the local geometric properties of the graphs of each function. Finally, we again apply the polar-coordinate transformation, and utilise Proposition~\ref{Prop_Star-shaped_eps-neighbourhood} to ensure that the resulting plane curve is the $\eps$-neighbourhood $E_\eps$ for some set $E \in \R^2$ and $\eps > 0$.
%
%

Before we embark on the construction, we recall the definition of the Cantor function. The following definition is from~\cite{DovMarRyaVuo2006}.
\begin{definition}[Cantor function] \label{Def_Cantor_function}
Let $x \in [0, 1]$ and expand $x$ as
\begin{equation} \label{Def_Cantor_expansion}
x = \sum_{n=1}^\infty \frac{a_n(x)}{3^n}, a_n(x) \in \{0,1,2\}.
\end{equation}
Denote by $N_x$ the smallest $n$ with $a_n(x) = 1$ if it exists and put $N_x = \infty$ if there is no such $a_n(x)$. Then the Cantor function $G : [0, 1] \to \R$ is given by the expression
\begin{equation} \label{Eq_Def_Cantor_function}
G(x) := \frac{1}{2^{N_x}} + \frac{1}{2} \sum_{n=1}^{N_x - 1} \frac{a_n(x)}{2^n}.
\end{equation}
\end{definition}
Let $M \in [0,1]$ be the set of points at which $G$ fails to have a finite or infinite derivative. It is known that the Hausdorff dimension of $M$ satisfies $\textrm{dim}_\mathcal{H}(M) = \lg 2 / \lg 3$, see~\cite[Theorem 8.17]{DovMarRyaVuo2006}. In particular, $M$ is an uncountable set.


\begin{figure}[h]
      \captionsetup{margin=0.75cm}
      \centering     
                \includegraphics[width=0.75\textwidth]{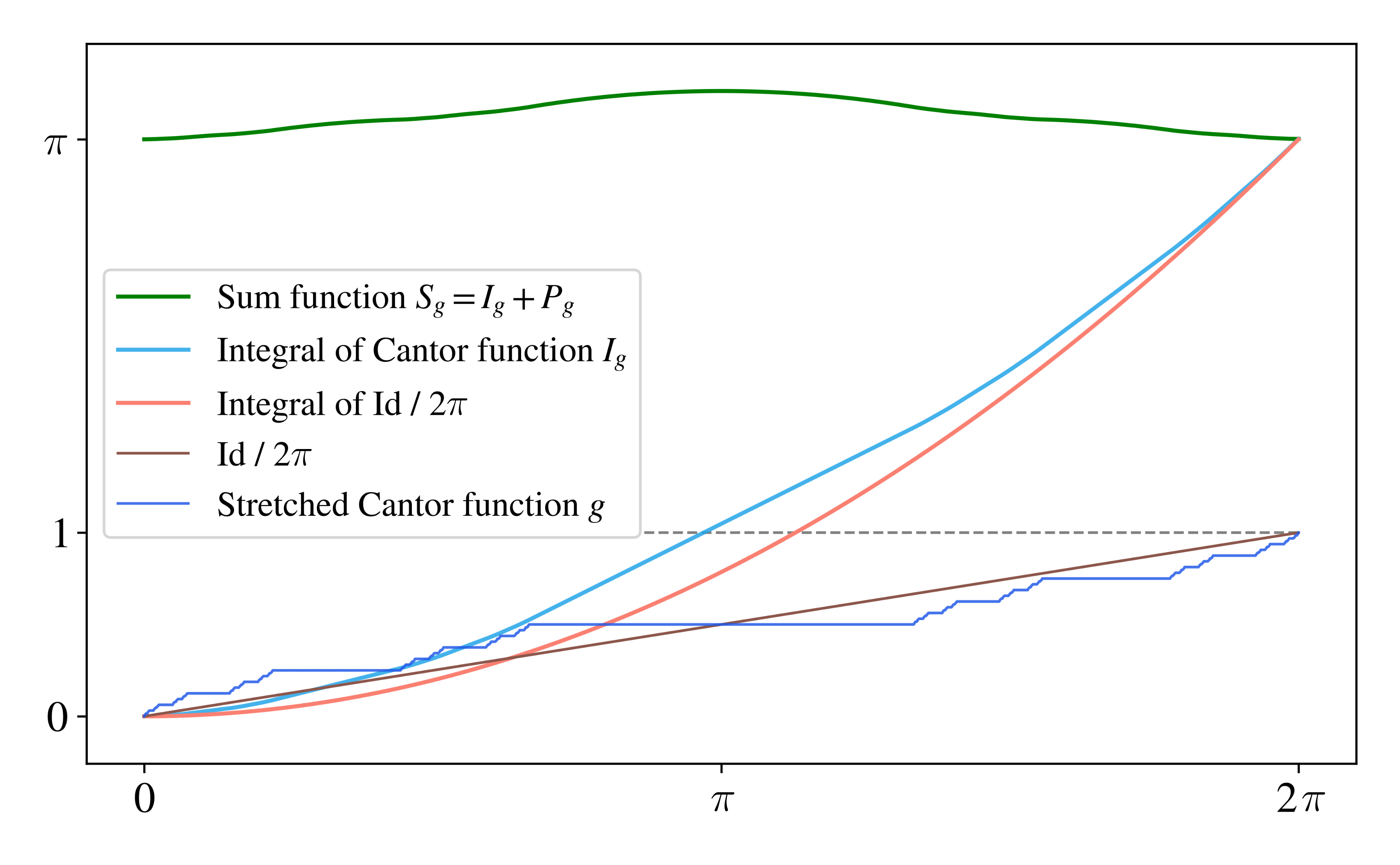} \\[-1mm]
                \caption{Illustration of the functions defined in the proof of Theorem~\ref{Thm_non-existence_of_curvature}.}
\label{Fig_Cantor_integral}
\end{figure}

\begin{theorem_without_2}[\textbf{Existence of a nowhere-smooth boundary with non-existing second derivative in an uncountable set}] \label{Thm_non-existence_of_curvature}
There exists a planar set $E \subset \R^2$ and $\eps > 0$ for which
\begin{itemize}
    \item[(i)] the boundary $\partial E_\eps$ consists only of wedges and shallow singularities, and
    \item[(ii)] there exists a subset $C \subset \partial E_\eps$ with positive Hausdorff dimension, on which curvature is not defined in the classical sense.
\end{itemize}
\end{theorem_without_2}

\begin{proof}
Let $G$ be the Cantor function~\eqref{Eq_Def_Cantor_function}, and define $g : [0, 2\pi] \to [0,1]$ by setting $g(x) := G(x / 2\pi)$. Let $I_g : [0,2\pi] \to \R$ be the integral function of $g$ with $I_g(x) := \int_0^x g(t) dt$. Then the second derivative $I_g''$ is not defined precisely in the subset $M' := \big\{ 2\pi x \, | \, x \in M\big\}$ for which $g$ fails to have the first derivative. Since the $k$-dimensional Hausdorff measure scales by a factor of $\lambda^k$ in homotheties defined by a scaling factor $\lambda$, it follows that $\textrm{dim}_\mathcal{H}(M') = \textrm{dim}_\mathcal{H}(M) = \lg 2 / \lg 3$.

The graph of the Cantor function $G$ possesses point symmetry relative to the symmetry center $\big(\frac{1}{2}, \frac{1}{2}\big)$. It is easy to observe from this that $I_g(x) \geq \frac{1}{4\pi}x^2 = \int_0^x \textrm{Id}(t/2\pi)dt$ for all $x \in [0, 2\pi]$ and that $I_g(2\pi) = \frac{1}{4\pi}(2\pi)^2 = \pi$.
Thus, by adding to $I_g$ the second degree polynomial $P_g(x) := -\frac{1}{4\pi}x^2 + \pi$ results in a semiconvex function $S_g := I_g + P_g$ for which $\min_{x\in [0, 2\pi]} S_g(x) = \pi$. See Figure~\ref{Fig_Cantor_integral} for illustrations of the functions $g$, $I_g$, and $S_g$.

Now, let $S$ be as in~\eqref{Eq_Def_S}. It was argued in the proof of Theorem~\ref{Thm_main_existence} that $S$ is semiconvex and that $\textrm{min}_{x \in [0, 2\pi]}S(x) > 0$. Since the sum of two convex functions is convex, we obtain that the sum function $V := S + S_g$ is semiconvex and satisfies $\textrm{min}_{x \in [0, 2\pi]}V(x) > \pi$. Thus, $\textrm{epi } V$ has positive reach due to~\cite[Theorem 2.3]{Fu_Tubular_neighborhoods}. It then follows from Proposition~\ref{Prop_Star-shaped_eps-neighbourhood} that the image $\alpha(\Gamma_V)$ under the polar-coordinate transformation $\alpha$ in~\eqref{Eq_Def_polar_map} coincides with the $\eps$-neighbourhood boundary $\partial E_\eps$ for some $E \subset \R^2$. Since $\alpha$ is smooth, it preserves the sets of (non)-existence of the first and second derivatives. Furthermore, because $\alpha$ is bi-Lipschitz, it preserves the Hausdorff dimension of these sets, see for instance~\cite[Proposition 2.49]{Ambrosio_et_al_Functions_of_Bounded_Variation}. From this, the claim follows.
\end{proof}

Finally we note that the construction in the proof of Theorem~\ref{Thm_non-existence_of_curvature} can be easily modified so that the subset on the boundary where curvature is not defined is dense on $\partial E_\eps$. This can be done by iteratively replacing the plateaus of the stretched Cantor function -- for example the constant value on the interval $\big(\frac{2\pi}{3}, \frac{4\pi}{3}\big)$ -- by smaller copies of the Cantor function, proceeding from the wider intervals to shorter ones. The replaced parts must be appropriately scaled so that the value of the resulting function remains finite at the end of the interval $[0, 2\pi]$. This can be ensured because the original function will only be modified in a countable collection of subintervals.

\section{Discussion}
We have shown that there exist planar $\eps$-neighbourhoods whose boundary has no $C^1$-smooth subsets, due to the existence of wedge singularities densely on the boundary. This result is sharp in the sense that there can generally only exist countably many such singularities on an $\eps$-neighbourhood boundary~\cite[Lemma 5.2]{lamb2025boundariesvarepsilonneighbourhoodsplanarsets}. We have also shown that it is possible to extend the previous example to obtain an $\eps$-neighbourhood boundary on which wedge singularities lie densely and which also contains a subset of positive Hausdorff dimension on which curvature is not defined in the classical sense. This clarifies the existing result that curvature is defined almost everywhere on the boundary with respect to the one-dimensional Hausdorff measure~\cite[Theorem 7]{lamb2025boundariesvarepsilonneighbourhoodsplanarsets}. 
An interesting direction for further study is to investigate whether such pathological geometries can naturally occur for example as boundaries of attractors in random dynamical systems with bounded noise. These attractors are by definition geometrically $\eps$-neighbourhoods, where $\eps$ has an interpretation as the maximum amplitude of additive noise in the system. Singularities on attractors are known to encode information about the dynamics in such systems.

\section*{Acknowledgments}
The authors gratefully  acknowledge support from EU Marie Sklodowska-Curie ITN Critical Transitions in Complex Systems (H2020-MSCA-ITN-2014 643073 CRITICS), which funded the PhD of KT. Furthermore, JSWL and MR were supported by the EPSRC (grants EP/W009455/1 and  EP/Y020669/1), and JSWL acknowledges support from the EPSRC Centre for Doctoral Training in Mathematics of Random Systems: Analysis, Modelling and Simulation (EP/S023925/1) and JST Moonshot R \& D Grant Number JPMJMS2021. KT gratefully acknowledges support from the European Research Council Advanced Grant (ERC AdG, ILLUSIVE: Foundations of Perception Engineering, 101020977).

\bibliographystyle{amsplain}

\bibliography{Set_Valued_Dynamics_Theory_1}{}

\end{document}